\PassOptionsToPackage{giveninits=true}{biblatex}
\documentclass[arxiv=true]{agn_article}
\usepackage{agn_all}
\usepackage{hyperref}
\usepackage{caption}
\newcommand{\sample}{343}
\newcommand{\graphRatio}{.63}

\newcommand{\Cm}{C_m}

\newcommand{\WADM}{W_{\mathrm{ADM}}}
\begin{document}
\title{Quasiconvex relaxation of isotropic functions in incompressible planar hyperelasticity}
\date{\today}
\author{%
	Robert J. Martin%
	\thanks{\;\;%
		Corresponding author: Robert J. Martin, Lehrstuhl f\"{u}r Nichtlineare Analysis und Modellierung, Fakult\"at f\"ur Mathematik, Universit\"at Duisburg--Essen, Campus Essen, Thea-Leymann Stra{\ss}e 9, 45141 Essen, Germany,
		email: \href{mailto:robert.martin@uni-due.de}{robert.martin@uni-due.de}%
	},
	\quad
	Jendrik Voss\thanks{\;\;%
		Jendrik Voss, Lehrstuhl f\"{u}r Nichtlineare Analysis und Modellierung, Fakult\"{a}t f\"{u}r Mathematik, Universit\"{a}t Duisburg--Essen, Thea-Leymann Str. 9, 45127 Essen, Germany;
		email: \href{mailto:max.voss@uni-due.de}{max.voss@uni-due.de}%
	},
	\quad
	Ionel-Dumitrel Ghiba\thanks{\;\;%
		Ionel-Dumitrel Ghiba, Alexandru Ioan Cuza University of Ia\c si, Department of Mathematics, Blvd.~Carol I, no.\,11, 700506 Ia\c si, Romania; and Octav Mayer Institute of Mathematics of the Romanian Academy, Ia\c si Branch, 700505 Ia\c si,
		email: \href{mailto:dumitrel.ghiba@uaic.ro}{dumitrel.ghiba@uaic.ro}%
	}
	~\quad and\quad%
	Patrizio Neff\thanks{\;\;%
		Patrizio Neff, Head of Lehrstuhl f\"{u}r Nichtlineare Analysis und Modellierung, Fakult\"at f\"ur Mathematik, Universit\"at Duisburg--Essen, Campus Essen, Thea-Leymann Stra{\ss}e 9, 45141 Essen, Germany,
		email: \href{mailto:patrizio.neff@uni-due.de}{patrizio.neff@uni-due.de}%
	}
}
\maketitle
\begin{abstract}
	\noindent
	In this note, we provide an explicit formula for computing the quasiconvex envelope of any real-valued function $W\col\SL(2)\to\R$ with $W(RF)=W(FR)=W(F)$ for all $F\in\SL(2)$ and all $R\in\SO(2)$, where $\SL(2)$ and $\SO(2)$ denote the special linear group and the special orthogonal group, respectively. In order to obtain our result, we combine earlier work by Dacorogna and Koshigoe on the relaxation of certain conformal planar energy functions with a recent result on the equivalence between polyconvexity and rank-one convexity for objective and isotropic energies in planar incompressible nonlinear elasticity.
\end{abstract}

\vspace*{1.4em}
\\[1.4em]
\tableofcontents
\newpage

\section{Introduction}

A classical task in nonlinear hyperelasticity is to minimize an energy functional of the form
\begin{equation}\label{eq:generalEnergyFunctional}
	I\col W^{1,p}(\Omega;\R^n)\to\R\,,\quad I(\varphi)=\int_\Omega W(\grad\varphi(x))\,\dx
\end{equation}
under certain boundary conditions, where $\Omega\subset\R^n$ represents the reference configuration of an elastic body. The elastic behaviour of the body is completely determined by the choice of a particular energy density $W$ depending on the deformation gradient $F=\grad\varphi$. In the compressible case, since the exclusion of (local) self-intersection implies $\det F>0$, the domain of the energy $W$ is restricted to the group $\GLpn$ of $n\times n$\.--\.matrices with positive determinant. Modeling deformations of incompressible materials \cite{fosdick1986minimization}, on the other hand, requires the stronger constraint $\det F=1$; in this case, the natural domain of the energy is given by the special linear group $\SLn$.

In order to ensure the existence of minimizers for functionals of the form \eqref{eq:generalEnergyFunctional}, it is necessary to pose additional conditions on the energy density $W$. The most common requirements for this purpose are certain \emph{generalized convexity properties}: Since classical convexity of $W$ leads to physically unreasonable material behaviour \cite{silhavy1997mechanics}, weakened notions of convexity are usually considered, the most important ones being \emph{rank-one convexity}, \emph{quasiconvexity} and \emph{polyconvexity}.

Compared to functions defined on the full matrix space $\Rnn$, the restricted domain of the energy $W$ poses additional challenges with respect to these convexity properties (a number of which were famously addressed and solved by John Ball in his seminal 1977 paper \cite{ball1976convexity,ball1977constitutive}), but also allow for obtaining some significantly simplified criteria. In particular, under the additional assumptions of objectivity and isotropy, a large number of necessary and sufficient criteria for rank-one convexity and polyconvexity of energy functions on $\GLpn$ and $\SLn$ have been given in the literature \cite{knowles1976failure,knowles1978failure,abeyaratne1980discontinuous,dacorogna1993different,rosakis1994,aubert1995necessary,silhavy1997mechanics,Dunn2003,mielke2005necessary,agn_martin2015rank}.

In the two-dimensional case of planar elasticity, the above generalized convexity properties can be simplified even further. In addition to the well-known observation that polyconvexity and convexity\footnote{%
	A function $W\col M\to\R$ on a non-convex domain $M\subset\Rnn$, e.g.\ on $M=\GLpn$ or $M=\SLn$, is called convex if there exists a convex function $\Wtilde\col\Rnn\to\R\cup\{+\infty\}$ with $\Wtilde(F)=W(F)$ for all $F\in M$, cf.\ Definition \ref{def:convexityPropertiesSL}.
}
of a function $W\col\SL(2)\to\R$ are equivalent,\footnote{%
	Under the constraint $\det F=1$, any polyconvex representation $W(F)=P(F,\det F)$ can be reduced to a convex function in terms of $F$, cf.\ \cite[Lemma B.1]{agn_ghiba2018rank}.%
}
it was recently demonstrated that in the planar incompressible case, these properties are in turn equivalent to the (generally weaker) rank-one convexity and quasiconvexity for isotropic and objective energy functions \cite{agn_ghiba2018rank}. Based on these earlier results, this note provides an explicit relaxation result which allows for a direct computation of the quasiconvex envelope of any isotropic and objective function $W\col\SL(2)\to\R$.

\section{\boldmath Generalized convexity properties of incompressible energy functions}

Apart from classical convexity, we will consider the following weakened convexity properties of planar energy functions with values in $\R\cup\{+\infty\}$.

\begin{definition}\label{def:convexityPropertiesGloballyDefinedFunctions}
	Let $W\col\R^{2\times2}\to\R\cup\{+\infty\}$. Then $W$ is called
	\begin{itemize}
		\item
			\emph{rank-one convex} if for all $F\in\R^{2\times2}$, $t\in[0,1]$ and $H\in\R^{2\times2}$ with $\rank(H)=1$,
			\begin{equation*}
				W((1-t)F+t(F+H)) \;\leq\; (1-t)\.W(F) + t\.W(F+H)\,;
			\end{equation*}
		\item
		\emph{quasiconvex} if for every bounded open set $\Omega\subset\R^2$ and all smooth functions $\vartheta\in C_0^\infty (\Omega)$ with compact support, %
		\begin{equation*}
			\int_{\Omega}W(F_0+\nabla \vartheta)\,\dx\geq \int_{\Omega}W(F_0)\,\dx=W(F_0)\cdot \abs{\Omega}\,;
		\end{equation*}
		\item
			\emph{polyconvex} if
			\begin{equation*}
				W(F) = P(F,\det F) \qquad\text{for some convex function }\;P\col\R^{2\times2}\times\R \;\cong\; \R^5\to\R\cup\{+\infty\}\,.
			\end{equation*}
	\end{itemize}
\end{definition}
For an \emph{incompressible} planar energy, i.e.\ a finite-valued function $W\col\SL(2)\to\R$ defined on the domain $\SL(2)$ only, we will employ the following definitions.

\begin{definition}\label{def:convexityPropertiesSL}
	Let $W\col\SL(2)\to\R$. Then $W$ is called rank-one convex [quasiconvex/polyconvex] if the function
	\[
		\What\col\R^{2\times2}\to\R\cup\{+\infty\}\,,\quad \What(F) =
		\begin{cases}
			W(F) &\col F\in\SL(2)\,,\\
			+\infty &\col F\notin\SL(2)\,.
		\end{cases}
	\]
	is rank-one convex [quasiconvex/polyconvex] in the sense of Definition \ref{def:convexityPropertiesGloballyDefinedFunctions}. Furthermore, $W$ is called convex if there exists a convex function $\Wtilde\col\R^{2\times2}\to\R\cup\{+\infty\}$ such that $\Wtilde(F)=W(F)$ for all $F\in\SL(2)$.
\end{definition}

Defining quasiconvexity for functions which may attain the value $+\infty$ is often avoided completely since, in this case, it no longer implies the weak lower semicontinuity of the associated energy functional \cite{Dacorogna08,ball1976convexity}. Furthermore, a quasiconvex function with values in $\R\cup\{+\infty\}$ is not necessarily rank-one convex in general \cite{Dacorogna08}. However, for the incompressible case considered here (i.e.\ $W(F)=+\infty$ if and only if $F\notin\SL(2)$), quasiconvexity of $W$ does indeed imply rank-one convexity, as shown by Conti \cite{conti2008quasiconvex}.

Note also that for $W\col\SL(2)\to\R$, the existence of a polyconvex representation $W(F)=P(F,\det F)$ can be reduced to the case where $P(F,d)=+\infty$ if and only if $d\neq1$. Thereby, $P$ is reduced to a (convex) function in terms of $F$, which implies that $W$ is polyconvex if and only if $W$ (or rather an extension of $W$ to the domain $\R^{2\times2}$) is convex (cf.\ \cite[Lemma B.1]{agn_ghiba2018rank}).

\subsection{Generalized convex envelopes}

For each of the convexity properties considered in the previous section, we can define a corresponding \emph{envelope} of an arbitrary function on $\R^{2\times2}$.

\begin{definition}\label{def:envelopesGloballyDefinedFunctions}
	Let $W\col\R^{2\times2}\to\R\cup\{+\infty\}$ be bounded below. Then the \emph{rank-one convex}, \emph{quasiconvex}, \emph{polyconvex} and \emph{convex} \emph{envelopes} $RW,QW,PW,CW\col\R^{2\times2}\to\R\cup\{+\infty\}$ of $W$ are respectively defined by
	\begin{alignat*}{2}
		RW(F) &= \sup \{ w(F) \setvert w\col\R^{2\times2}\to\R\cup\{+\infty\} \,\text{ rank-one convex, }\;& w(X)&\leq W(X) \,\text{ for all }X\in\R^{2\times2} \}\,,\\
		QW(F) &= \sup \{ w(F) \setvert w\col\R^{2\times2}\to\R\cup\{+\infty\} \,\text{ quasiconvex, }\;& w(X)&\leq W(X) \,\text{ for all }X\in\R^{2\times2} \}\,,\\
		PW(F) &= \sup \{ w(F) \setvert w\col\R^{2\times2}\to\R\cup\{+\infty\} \,\text{ polyconvex, }\;& w(X)&\leq W(X) \,\text{ for all }X\in\R^{2\times2} \}\,,\\
		CW(F) &= \sup \{ w(F) \setvert w\col\R^{2\times2}\to\R\cup\{+\infty\} \,\text{ convex, }\;& w(X)&\leq W(X) \,\text{ for all }X\in\R^{2\times2} \}\,.
	\end{alignat*}
\end{definition}
\noindent
Again, these definitions can be applied to functions defined on $\SL(2)$ via the natural extension of the domain to $\R^{2\times2}$.

\begin{definition}\label{def:envelopesSL}
	Let $W\col\SL(2)\to\R$. Then the \emph{rank-one convex}, \emph{quasiconvex}, \emph{polyconvex} and \emph{convex} \emph{envelope} of $W$ are defined by
	\begin{align*}
		RW = (R\What)\big|_{\SL(2)}\,,\qquad
		QW = (Q\What)\big|_{\SL(2)}\,,\qquad
		PW = (P\What)\big|_{\SL(2)}\,,\qquad
		CW = (C\What)\big|_{\SL(2)}
	\end{align*}
	in the sense of Definition \ref{def:envelopesGloballyDefinedFunctions}, where
	\[
		\What\col\R^{2\times2}\to\R\cup\{+\infty\}\,,\quad \What(F) =
		\begin{cases}
			W(F) &\col F\in\SL(2)\,,\\
			+\infty &\col F\notin\SL(2)\,.
		\end{cases}
	\]
\end{definition}

\begin{remark}\label{remark:envelopesInequality}
	The implications
	\[
		W\;\text{ convex }
		\quad\implies\quad
		W\;\text{ polyconvex }
		\quad\implies\quad
		W\;\text{ quasiconvex }
		\quad\implies\quad
		W\;\text{ rank-one convex,}
	\]
	which hold for any function $W\col\SL(2)\to\R$ (cf.\ \cite{Dacorogna08,conti2008quasiconvex}), immediately imply the inequalities
	\[
		CW(F) \leq PW(F) \leq QW(F) \leq RW(F) \qquad\text{for all }\;F\in\SL(2)\,.
	\]
\end{remark}

The quasiconvex envelope, in particular, plays an important role in relaxation approaches to non-quasiconvex minimization problems: If, for an energy of the form \eqref{eq:generalEnergyFunctional}, the existence of minimizers under boundary conditions cannot be ensured, then the \emph{infimum} of the attained energy values might in many cases be obtained instead by minimizing the relaxed functional \cite[Chapter~9]{Dacorogna08}
\[
	I\col W^{1,p}(\Omega;\R^n)\to\R\,,\quad I(\varphi)=\int_\Omega QW(\grad\varphi(x))\,\dx\,.
\]
Such relaxation methods are used, for example, in the modeling of materials with complex microstructures \cite{kinderlehrer2012microstructure,conti2015analysis}.

In general, computing the quasiconvex envelope of a given energy $W$ is a rather difficult problem, with explicit representations being available only for a small number of special cases \cite{dacorogna1993different,ledret1995quasiconvex}. The main result of this note (Theorem \ref{theorem:convexEnvelopes}), however, shows that in the objective and isotropic case of planar incompressible energies, this task can be accomplished by simple analytical methods.

\section{\boldmath The quasiconvex envelope of objective and isotropic functions on $\SL(2)$}

It is well known that any objective and isotropic function $W\col\GLp(n)\to\R$ can be expressed in terms of singular values,\footnote{In nonlinear elasticity, the singular values of the deformation gradient $F\in\GLp(n)$, which coincide with the eigenvalues of both the material stretch tensor $U=\sqrt{F^TF}$ and the spatial stretch tensor $V=\sqrt{FF^T}$, are also called \emph{principal stretches}.} i.e.\ there exists a symmetric function $q\col(0,\infty)^n\to\R$ such that $W(F)=q(\lambda_1,\dotsc,\lambda_n)$ for all $F\in\GLpn$ with singular values $\lambda_1,\dotsc,\lambda_n$. The corresponding representation $W(F)=q(\lambda_1,\lambda_2)$ of a planar incompressible energy $W$ can be simplified even further.

\begin{lemma}\label{lemma:representationsSL}
	Let $W\col\SL(2)\to\R$ be an objective and isotropic function. Then there exist uniquely defined functions $q\col(0,\infty)\times(0,\infty)\to\R$,\; $\phi\col[0,\infty)\to\R$ and $\widetilde{\phi}\col\R\to\R$ such that for all $F\in\SL(2)$ with singular values $\lambda_1,\lambda_2$,
	\begin{equation}\label{eq:energyRepresentationSingularValueDifference}
		W(F) = q(\lambda_1,\lambda_2) = \widetilde{\phi}(\lambda_1-\lambda_2) = \phi\left(\lambdamax(F)-\frac{1}{\lambdamax(F)}\right)\,,
	\end{equation}
	where $\lambdamax(F)=\max\{\lambda_1,\lambda_2\}$.
\end{lemma}

Note that $q(x,y)=q(y,x)$ for all $x,y>0$ and $\widetilde{\phi}(-t)=\widetilde{\phi}(t)$ for all $t\in\R$ due to the isotropy of $W$. Furthermore, it is easy to see that for any real-valued function $\phi\col[0,\infty)\to\R$ or any $\widetilde{\phi}\col\R\to\R$ with $\widetilde{\phi}(-t)=\widetilde{\phi}(t)$ for all $t\in\R$, an objective and isotropic energy $W\col\SL(2)\to\R$ is defined by \eqref{eq:energyRepresentationSingularValueDifference}.

Different representations, for example in terms of the squared Frobenius matrix norm $\norm{F}^2=\sum_{i,j=1}^2F_{ij}^2$, have been considered in the literature as well \cite{abeyaratne1980discontinuous} (cf.\ \cite{Dunn2003}). However, expressing $W$ in the form \eqref{eq:energyRepresentationSingularValueDifference} allows for stating convexity criteria in particularly simple terms (cf.\ Theorem \ref{theorem:convexityCriterion}).

\medskip

In view of Lemma \ref{lemma:representationsSL}, the equality
\begin{equation}\label{eq:lambdaDifferenceDetRepresentation}
	\abs{\lambda_1-\lambda_2}
	= \sqrt{\norm{F}-2} = \sqrt{\norm{F}-2\.\det F}\,,
\end{equation}
which holds for any $F\in\SL(2)$, suggests a direct connection between the notion of convex envelopes in the incompressible case and an earlier relaxation result by Dacorogna and Koshigoe \cite{dacorogna1993different}.

\begin{proposition}[{\cite[Proposition 5.1]{dacorogna1993different}, cf.~\cite{vsilhavy2001rank}}]
\label{prop:dacorogna}
	Let $\Wbar\col\R^{2\times2}\to\R$ be of the form
	\begin{equation}
	\label{eq:dacorognaConformalSpecialCase}
		\Wbar\col\R^{2\times 2}\to\R\,,\quad \Wbar(F)= g(\sqrt{\norm{F}^2-2\.\det F})
	\end{equation}
	for some $g\col[0,\infty)\to\R$. Then
	\begin{equation}\label{eq:dacorognaResult}
		R\Wbar(F)=Q\Wbar(F)=P\Wbar(F)=C\Wbar(F)=\gtilde^{**}\bigl(\sqrt{\norm{F}^2-2\.\det F}\.\bigr)\,,
	\end{equation}
	where $\gtilde^*$ is the Legendre-transformation of the extended function
	\[
		\gtilde\col\R\to\R\,,\quad \gtilde(x)=
		\begin{cases}
		\hfill g(x) &:\; x\geq0\\
		g(-x) &:\; x< 0
		\end{cases}
	\]
	and $\gtilde^{**}=\left(\gtilde^*\right)^*$.	
\end{proposition}
\begin{remark}\label{remark:biconjugateConvexEnvelope}
	Since $\gtilde$ is finite valued on $\R$, the equality $\gtilde^{**}=C\gtilde$ between the biconjugate and the convex envelope of $\gtilde$ holds if $\gtilde$ is bounded below \cite[Theorem 2.43]{Dacorogna08}.
\end{remark}

Due to \eqref{eq:lambdaDifferenceDetRepresentation}, the restriction $W=\Wbar\big|_{\SL(2)}$ of any function $\Wbar\col\R^{2\times2}\to\R$ of the form \eqref{eq:dacorognaConformalSpecialCase} to $\SL(2)$ can be written as
\[
	W(F) = \Wbar(F) = g(\abs{\lambda_1-\lambda_2}) = \gtilde(\lambda_1-\lambda_2)
\]
for $F\in\SL(2)$, i.e.\ in the form \eqref{eq:energyRepresentationSingularValueDifference} with $\phi=g$ and $\widetilde{\phi}=\gtilde$. Similarly, any objective and isotropic $W\col\SL(2)\to\R$ can be uniquely extended to a function $\Wbar\col\R^{2\times2}\to\R$ of the form \eqref{eq:dacorognaConformalSpecialCase} by letting $g=\phi$ or, equivalently, $\gtilde=\widetilde{\phi}$.

\medskip

However, despite this striking connection, Proposition \ref{prop:dacorogna} is not immediately applicable to the case of functions defined on $\SL(2)$: Note carefully that the convex envelopes in eq.\ \eqref{eq:dacorognaResult} take into account not only the value of $W$ on $\SL(2)$, but also the value of a \emph{specific} extension $\What$ of $W$ to $\R^{2\times2}$. The underlying difference is that the notion of (generalized) convexity on a subset of $\R^{2\times2}$ (cf.\ Definition \ref{def:convexityPropertiesSL}) requires $W$ to have \emph{any} extension to $\R^{2\times2}$ satisfying the respective convexity property, which is, a priori, not necessarily of the form \eqref{eq:dacorognaConformalSpecialCase}.

On the other hand, Proposition \ref{prop:dacorogna} can be used to obtain lower bounds for the envelopes of incompressible energies; particularly,
\begin{alignat}{2}
	C\Wbar(F) &= \sup \{ w(F) \setvert w\col\R^{2\times2}\to\R \,\text{ convex, }\;& w(X)&\leq \Wbar(X) \,\text{ for all }X\in\R^{2\times2} \}\nonumber\\
	&\leq \sup \{ w(F) \setvert w\col\R^{2\times2}\to\R \,\text{ convex, }\;& w(X)&\leq \Wbar(X) \,\text{ for all }X\in\SL(2) \} \label{eq:convexEnvelopeExtensionInequality}\\
	&= \sup \{ w(F) \setvert w\col\R^{2\times2}\to\R \,\text{ convex, }\;& w(X)&\leq W(X) \,\text{ for all }X\in\SL(2) \}
	\;=\; CW(F) \nonumber
\end{alignat}
for any objective and isotropic function $W\col\SL(2)\to\R$ and all $F\in\SL(2)$, where $\Wbar\col\R^{2\times2}\to\R$ denotes the extension of $W$ described above. Again, note carefully that it is not immediately obvious whether equality holds in \eqref{eq:convexEnvelopeExtensionInequality}.

In order to fully establish a result similar to Proposition \ref{prop:dacorogna} in the incompressible case, we will require the following criteria for generalized convexity properties.

\begin{theorem}[\cite{agn_ghiba2018rank}]\label{theorem:convexityCriterion}
	Let $W\col\SL(2)\to\R$ be an objective and isotropic function. Then the following are equivalent:
	\begin{itemize}
		\item[i)] $W$ is rank-one convex,
		\item[ii)] $W$ is polyconvex,
		\item[iii)] the function $\widetilde{\phi}\col\R\to\R$ with $W(F) = \widetilde{\phi}(\lambda_1-\lambda_2)$ for all $F\in\SL(2)$ with singular values $\lambda_1,\lambda_2$ is convex,
		\item[iv)] the function $\phi\col[0,\infty)\to\R$ with $W(F) = \phi(\sqrt{\norm{F}^2-2})=\phi\Bigl(\lambdamax(F)-\displaystyle\frac{1}{\lambdamax(F)}\Bigr)$ is nondecreasing and convex.
	\end{itemize}
\end{theorem}
The characterization of polyconvex energies on $\SL(2)$ by criterion iv) in Theorem \ref{theorem:convexityCriterion} is originally due to Mielke \cite{mielke2005necessary}. A criterion for the rank-one convexity of a twice differentiable energy in terms of the representation $W(F)=\Psi(\norm{F}^2)=\Psi(\lambda_1^2+\lambda_2^2)$ has previously been given by Abeyaratne \cite{abeyaratne1980discontinuous}.

Using Theorem \ref{theorem:convexityCriterion}, it is possible to find an explicit representation of the generalized convex envelopes $RW$, $QW$, $PW$ and $CW$ for any isotropic and objective function on $\SL(2)$.

\begin{theorem}\label{theorem:convexEnvelopes}
	Let $W\col\SL(2)\to\R$ be objective, isotropic and bounded below. Then
	\begin{equation}\label{eq:convexEnvelopesFormula}
		RW(F) = QW(F) = PW(F) = CW(F) = C\widetilde{\phi}(\lambda_1-\lambda_2) = \Cm\phi\left(\lambdamax(F)-\displaystyle\frac{1}{\lambdamax(F)}\right)
	\end{equation}
	for all $F\in\SL(2)$ with singular values $\lambda_1,\lambda_2$, where $\lambdamax(F)=\max\{\lambda_1,\lambda_2\}$ and $\Cm\phi\col[0,\infty)\to\R$ denotes the \emph{monotone-convex envelope} of $\phi$, given by
	\[
		\Cm\phi(t) \colonequals \sup\Big\{p(t) \setvert p\col[0,\infty)\to\R \text{ monotone increasing and convex with } p(s)\leq \phi(s)\;\forall\,s\in[0,\infty)\Big\}\,,
	\]
	i.e.\ the largest monotone and convex function bounded above by $\phi$.
\end{theorem}
\begin{proof}
	Since $\widetilde{\phi}(-t)=\widetilde{\phi}(t)=\phi(t)$ for all $t\geq0$, it is easy to see (cf.\ \cite{dacorogna1993different}) that $\Cm\phi(t)=C\widetilde{\phi}(t)=C\widetilde{\phi}(-t)$ for all $t\geq0$ and thus, in particular,
	\[
		C\widetilde{\phi}(\lambda_1-\lambda_2) = \Cm\phi(\abs{\lambda_1-\lambda_2}) = \Cm\phi\left(\lambdamax(F)-\frac{1}{\lambdamax(F)}\right)
	\]
	for all $\lambda_1,\lambda_2>0$. Furthermore, Remark \ref{remark:envelopesInequality} establishes the inequalities $CW(F) \leq PW(F) \leq QW(F) \leq RW(F)$, thus it remains to show that $RW(F)\leq C\widetilde{\phi}(\lambda_1-\lambda_2)\leq CW(F)$.
	
	According to Lemma \ref{lemma:representationsSL}, there exists a uniquely determined $\widetilde{\psi}\col\R\to\R$ such that $RW(F)=\widetilde{\psi}(\lambda_1-\lambda_2)$ for all $F\in\SL(2)$ with singular values $\lambda_1,\lambda_2$.\footnote{%
		Note that the rank-one convex envelope of an objective and isotropic function is itself objective and isotropic \cite{buttazzo1994envelopes}.%
	}
	Due to the rank-one convexity of $RW$ and Theorem~\ref{theorem:convexityCriterion}, the function $\widetilde{\psi}$ is convex. Since
	\[
		\widetilde{\psi}(t) = RW\Bigl(\diag\bigl(\tfrac{\sqrt{4+t^2}+t}{2},\tfrac{\sqrt{4+t^2}-t}{2}\bigr)\Bigr) \leq W\Bigl(\diag\bigl(\tfrac{\sqrt{4+t^2}+t}{2},\tfrac{\sqrt{4+t^2}-t}{2}\bigr)\Bigr) = \widetilde{\phi}(t)
	\]
	as well, we find $\widetilde{\psi}(t)\leq C \widetilde{\phi}(t)$ for all $t\in\R$ and thus
	\[
		RW(F) = \widetilde{\psi}(\lambda_1-\lambda_2) \leq C\widetilde{\phi}(\lambda_1-\lambda_2)
	\]
	for all $F\in\SL(2)$.
	
	Now, in order to establish the remaining inequality $C\widetilde{\phi}(\lambda_1-\lambda_2)\leq CW(F)$, let
	\[
		\Wbar\col\R^{2\times2}\to\R\,,\quad \Wbar(F)=\widetilde{\phi}(\sqrt{\norm{F}-2\.\det F})
	\]
	denote the unique extension of $W$ to $\R^{2\times2}$ of the form \eqref{eq:dacorognaConformalSpecialCase}. Then using \eqref{eq:convexEnvelopeExtensionInequality} and Remark \ref{remark:biconjugateConvexEnvelope}, we find
	\[
		CW(F)\geq C\Wbar(F) = \widetilde{\phi}^{**}(\sqrt{\norm{F}-2\.\det F}) = \widetilde{\phi}^{**}(\lambda_1-\lambda_2) = C\widetilde{\phi}(\lambda_1-\lambda_2)
	\]
	for all $F\in\SL(2)$ with singular values $\lambda_1,\lambda_2$.
\end{proof}
Another result similar to Theorem \ref{theorem:convexEnvelopes} has previously been obtained \cite{agn_martin2019envelope} for so-called \emph{conformally invariant} functions on $\GLp(2)$, i.e.\ any $W\col\GLp(2)\to\R$ satisfying
\[
	W(A\.F\.B) = W(F) \qquad\text{for all }\; A,B\in\{a\.R\in\GLp(2) \setvert a\in(0,\infty)\,,\; R\in\SO(2)\}\,,
\]
where $\SO(2)$ denotes the special orthogonal group. In addition to being objective and isotropic, such a function is \emph{isochoric}, i.e.\ invariant under (purely volumetric) scaling of the deformation gradient $F$.
\begin{remark}
	Due to eq.\ \eqref{eq:convexEnvelopesFormula}, the problem of finding the quasiconvex (as well as the rank-one convex and the polyconvex) envelope of $W$ reduces to the task of computing the convex envelope of a scalar function. This latter problem can, for example, be solved by using Maxwell's equal area rule \cite[p.~319]{silhavy1997mechanics} and often admits a direct analytical solution.
\end{remark}

Combining Theorem \ref{theorem:convexEnvelopes} with Proposition \ref{prop:dacorogna} also yields the following relation between the envelopes of incompressible energies and their extensions to $\R^{2\times2}$ of the form $\eqref{eq:dacorognaConformalSpecialCase}$; recall from Remark \ref{remark:biconjugateConvexEnvelope} that $C\widetilde{\phi}=\widetilde{\phi}^{**}$ if $\widetilde{\phi}\col\R\to\R$ is bounded below.

\begin{corollary}
\label{cor:dacorognaComparison}
	Let $W\col\SL(2)\to\R$ be objective, isotropic and bounded below. Define
	\[
		\Wbar\col\R^{2\times2}\to\R\,,\quad \Wbar(F) = \widetilde{\phi}(\sqrt{\norm{F}^2-2\.\det F})\,,
	\]
	where $\widetilde{\phi}\col\R\to\R$ is the uniquely determined function with $W(F) = \widetilde{\phi}(\lambda_1-\lambda_2)$ for all $F\in\SL(2)$ with singular values $\lambda_1,\lambda_2$. Then
	\begin{align*}
		RW=QW=PW=CW=C\widetilde{\phi}(\lambda_1-\lambda_2)&=\widetilde{\phi}^{**}(\sqrt{\norm{F}^2-2\.\det F})\\
		&= (R\Wbar)\big|_{\SL(2)}=(Q\Wbar)\big|_{\SL(2)}=(P\Wbar)\big|_{\SL(2)}=(C\Wbar)\big|_{\SL(2)}\,.
	\end{align*}
\end{corollary}

\medskip

As a simple example, we consider the restriction of the classical \emph{Alibert–Dacorogna–Marcellini} energy \cite{alibert1992example,dacorogna1988counterexample}
\[
	\WADM\col\R^{2\times2}\to\R\,,\quad \WADM(F) = \norm{F}^2\.(\norm{F}^2-2\gamma\.\det F)\,,\qquad \gamma\in\R
\]
to the special linear group $\SL(2)$, i.e.
\[
	W\col\SL(2)\to\R\,,\quad W(F) = \norm{F}^4-2\gamma\.\norm{F}^2\,,\qquad \gamma\in\R\,,
\]
where $\norm{\,.\,}$ denotes the Frobenius norm. It was shown by Alibert, Dacorogna and Marcellini \cite{alibert1992example,dacorogna1988counterexample} that different convexity properties hold for $\WADM$ depending on the exact value of $\gamma$, which strictly distinguishes convexity ($\abs{\gamma}\leq\frac{2\.\sqrt{2}}{3}$), polyconvexity ($\abs{\gamma}\leq1$) and rank-one convexity ($\abs{\gamma}\leq\frac{2}{\sqrt{3}}$); the question whether $\WADM$ is \emph{not} quasiconvex for some $\abs{\gamma}\leq\frac{2}{\sqrt{3}}$ is still open \cite{dacorogna1998some}.

In the incompressible case, of course, the energy is simplified considerably; in particular, $W$ is convex for any $\gamma\leq2$ as the composition of a monotone and convex function with the convex mapping $F\mapsto\norm{F}$. In general, since (using the equality $\lambda_1\lambda_2=1$) we find
\begin{align*}
	W(F) &= (\lambda_1^2+\lambda_2^2)^2-2\gamma\.(\lambda_1^2+\lambda_2^2)\\
	&= (\lambda_1-\lambda_2)^4 + (4-2\.\gamma)\cdot (\lambda_1-\lambda_2)^2 + 4 - 4\.\gamma
\end{align*}
for all $F\in\SL(2)$ with singular values $\lambda_1,\lambda_2$, the function $W$ can be expressed as $W(F)=\widetilde{\phi}(\lambda_1-\lambda_2)$ with
\[
	\widetilde{\phi}(t) = t^4 + (4-2\.\gamma)\cdot t^2 + 4 - 4\.\gamma\,.
\]
Then $\widetilde{\phi}''(t) = 12\.t^2 + 8 - 4\.\gamma$ is nonnegative for all $t\geq0$ if and only if $\gamma\leq2$, thus according to Theorem \ref{theorem:convexityCriterion}, $W$ is rank-one convex, quasiconvex polyconvex and/or convex on $\SL(2)$ \emph{only} if $\gamma\leq2$. For $\gamma>2$, we can explicitly compute the convex envelope (cf.\ Figure \ref{fig:examplePolynomialEnvelope})
\[
	C\widetilde{\phi}(t)=
	\begin{cases}
		\widetilde{\phi}(t) &: t^2\geq \gamma-2\,,\\
		-\gamma^2 &: t^2< \gamma-2\,.
	\end{cases}
\]
of $\widetilde{\phi}$ and use Theorem \ref{theorem:convexEnvelopes} to find the generalized convex envelopes of $W$, which are given by
\begin{align*}
	RW(F) = QW(F) = PW(F) = CW(F) = C\widetilde{\phi}(\lambda_1-\lambda_2)
	&=
	\begin{cases}
		W(F) &: (\lambda_1-\lambda_2)^2 \geq \gamma-2\\
		-\gamma^2 &: (\lambda_1-\lambda_2)^2 < \gamma-2
	\end{cases}
	\\&=
	\begin{cases}
		W(F) &: \norm{F}^2 \geq \gamma\\
		-\gamma^2 &: \norm{F}^2 < \gamma
	\end{cases}
\end{align*}
for any $F\in\SL(2)$ with singular values $\lambda_1,\lambda_2=\frac{1}{\lambda_1}$.

\begin{figure}[h!]
	\begin{center}
		\begin{tikzpicture}
		\def\gammacon{11}
		\def\extcon{3}
		\def\yoffset{73.5}
			\begin{axis}[
			axis x line=middle,axis y line=middle,
			x label style={at={(current axis.right of origin)},anchor=north, below},
			xlabel=$t$, ylabel={},
			xmin=-4.41, xmax=4.41,
			ymin=-53.9,
			ymax=53.9,
			width=1\linewidth,
			height=\graphRatio\linewidth,
			ytick=\empty,
			xtick={-\extcon,\extcon},
			xticklabels={$-\sqrt{\gamma-2}$,$\sqrt{\gamma-2}$},
			hide obscured x ticks=false,
	        width=.91\linewidth,
	        height=.42\linewidth
			]
				\addplot[black, smooth][domain=-4.9:4.9, samples=\sample]{x^4+(4-2*\gammacon)*x^2+4-4*\gammacon+\yoffset} node[pos=.49, above left] {$\widetilde{\phi}$};
				\addplot[blue, dashed, very thick][domain=-4.9:-3, samples=\sample]{x^4+(4-2*\gammacon)*x^2+4-4*\gammacon+\yoffset};
				\addplot[blue, dashed, very thick] coordinates { ({-\extcon}, {-\gammacon^2+\yoffset})({\extcon},{-\gammacon^2+\yoffset})} node[pos=.63, above right]{$C\widetilde{\phi}$};
				\addplot[blue, dashed, very thick][domain=3:4.9, samples=\sample]{x^4+(4-2*\gammacon)*x^2+4-4*\gammacon+\yoffset};
			\end{axis}
		\end{tikzpicture}
	\end{center}
	\caption{%
	\label{fig:examplePolynomialEnvelope}%
		The convex envelope $C\widetilde{\phi}$ of $\widetilde{\phi}\col\R\to\R$ with $\widetilde{\phi}(t) = t^4 + (4-2\.\gamma)\cdot t^2 + 4 - 2\.\gamma$.%
	}
\end{figure}

\medskip

\noindent A further classical example of an elastic energy applicable to the incompressible case is given by the logarithmic \emph{Hencky strain energy} \cite{Hencky1928,Hencky1929,agn_neff2015geometry,agn_neff2015exponentiatedII,agn_ghiba2015exponentiated}
\[
	\WH\col\SL(2)\to\R\,,\quad \WH(F) = \norm{\log V}^2 = \norm{\log\sqrt{FF^T}}^2 = \log^2(\lambda_1)+\log^2(\lambda_2)\,,
\]
where $\log V = \log\sqrt{FF^T}$ denotes the principal matrix logarithm of the stretch tensor $V=\sqrt{FF^T}$. Note that for $\det F = 1$, $\WH$ can equivalently be expressed as
\[
	\norm{\dev\log V}^2 = \norm{\log ((\det V)^{-\afrac12}\.V)}^2 = \norm{\log ((\det F)^{-\afrac12}\.V)}^2 = \norm{\log V}^2 = \WH(F)\,,
\]
where $\dev X = X-\frac{\tr X}{2}\.\id$ is the deviatoric part of $X\in\R^{2\times2}$ and $\id$ denotes the identity matrix. Since
\[
	\WH(F) = \log^2(\lambda_1)+\log^2(\lambda_2) = 2\.\log^2(\lambdamax(F))
	= 2\.\log^2 \left( \frac{\abs{\lambda_1-\lambda_2}+\sqrt{4+(\lambda_1-\lambda_2)^2}}{2} \right)
	\,,
\]
the representation $\WH(F)=\widetilde{\phi}(\lambda_1-\lambda_2)$ of the Hencky energy is given by
\[
	\widetilde{\phi}(t) = 2\.\log^2\left(\frac{\abs{t}+\sqrt{4+t^2}}{2}\right)\,.
\]
Due to the sublinear growth of $\widetilde{\phi}$, we find $C\widetilde{\phi}\equiv0$ and thus, using Theorem \ref{theorem:convexEnvelopes},
\[
	R\WH=Q\WH=P\WH=C\WH\equiv0\,.
\]

\section{References}
\printbibliography[heading=none]

\end{document}